\documentclass{ams-l}

\newtheorem{theorem}{Theorem}[section]
\newtheorem{lemma}[theorem]{Lemma}
\newtheorem{corollary}[theorem]{Corollary}
\newtheorem{proposition}[theorem]{Proposition}
\newtheorem{claim}[theorem]{Claim}
\newtheorem{notation}{Notation}

\theoremstyle{definition}
\newtheorem{definition}[theorem]{Definition}

\theoremstyle{remark}
\newtheorem{remark}[theorem]{Remark}

\numberwithin{equation}{section}

\newcommand{\lcb}{\left [}
\newcommand{\rcb}{\right ]}

\newcommand{\te}{\rightarrow}

\begin{document}

\title[Schiffer's conjecture]{Some results on the Schiffer's conjecture in $R^{2}$}

\author{Jian Deng}
\address{CEMA, Central University of Finance and
Economics, Beijing, P.R.China, 100085}
\email{jdeng@fudan.edu.cn}


\date{\today}



\begin{abstract}
Let $\Omega$ be an open, bounded domain in $R^2$ with connected and $C^{\infty}$ boundary, and $\omega$ a solution of
\begin{eqnarray}
  -\triangle \omega &=& \mu \omega \\
  \frac{\partial \omega}{\partial n}|_{\partial \Omega} &=& 0\\
  \omega|_{\partial \Omega} &=& const \neq 0
\end{eqnarray}

\noindent for some $\mu > 0$. Denoting $0 = \mu_1(\Omega) < \mu_2(\Omega) \le \dots$ the set of all Neumann eigenvalues for the Laplacian on $\Omega$. We show that 1) if $\mu < \mu_8(\Omega)$; or 2) if $\Omega$ is strictly convex and centrally symmetric, $\mu < \mu_{13}(\Omega)$, then $\Omega$ must be a disk.
\end{abstract}

\maketitle


\section{Introduction}

Schiffer's conjecture (cf.Yau \cite{Yau}) is a long standing problem in spectral theory related to the Neumann eigenvalues of the Laplace operator. It is stated as follows:

Let $\Omega \subset R^2$ be a bounded domain. Does the existence of a nontrivial solution $u$ of the
over-determined Neumann eigenvalue problem

\begin{eqnarray}\label{pomp}
  -\triangle \omega &=& \mu \omega, \hspace{.1in}¡¡\mu > 0 \\ \label{pomp2}
  \frac{\partial \omega}{\partial n}|_{\partial \Omega} &=& 0\\ \label{pomp3}
  \omega|_{\partial \Omega} &=& const \neq 0
\end{eqnarray}

\noindent imply that $\Omega$ is a ball?

This problem is closely related to the Pompeiu problem in integral geometry. A domain $\Omega \subset R^2$ is said to
have the \emph{Pompeiu property} if and only if the only continuous function $f$ on $R^2$ for
which $\int_{\sigma(\Omega)} f(x,y) dxdy = 0$ for every rigid motion $\sigma$ of $R^2$ is the function $f \equiv 0$.
The connection between the Schiffer's conjecture and the Pompeiu property of domain $\Omega$ was established in (\cite{BST}), by showing that the failure of the Pompeiu property is equivalent to the existence of a nontrivial solution of (\ref{pomp}-\ref{pomp3}). Another remarkable result concerning the regularity of boundary of $\Omega$ was given in 1981 by Williams (\cite{W}). He proved that if a bounded Lipschitz domain $\Omega \subset R^n$ has a
connected boundary $\partial \Omega$, and if $\Omega$ fails to have the Pompeiu property, then $\partial \Omega$ is real analytic. Also, Berenstein (\cite{B}) proved that in $R^2$ the disk can be characterized
as the only simply-connected domain with $C^{2, \eta}$ boundary for which there exist infinitely many
solutions that solve (\ref{pomp}-\ref{pomp3}), and in (\cite{BY}) it was shown that for the system (\ref{pomp}-\ref{pomp3}) if $\mu = \mu_2(\Omega)$, the first positive Neumann eigenvalue of the Laplacian, then $\Omega$ is a ball. In (\cite{E1}, \cite{E2}) Ebenfelt considered the case where the boundary $\partial \Omega$ can be characterized as the image of a rational map of unit disk, among others, and he showed that $\Omega$ then must be a disk. The reader is referred to the beautiful survey of Zalcman (\cite{Z}) for an extensive exposition on the current state of the Pompeiu problem.

\bigskip


Our approach to Schiffer's conjecture is to estimate the dimension of the subspace of $H^1(\Omega)$ on which the bilinear form $$B(\phi, \psi; \mu) = \int\int_{\Omega} \frac{\partial \phi}{\partial x} \cdot  \frac{\partial \bar \psi}{\partial x} + \frac{\partial \phi}{\partial y} \cdot  \frac{\partial \bar\psi}{\partial y}  - \mu \phi \cdot \bar\psi \hspace{.1in} dxdy, \hspace{.1in} \phi, \psi \in H^1(\Omega)$$  is semi-negative definite. The subspace will be provided by the functions induced by symmetry. Specifically, we note that if $u$ is a smooth solution of $-\triangle v = \mu v$, then $\frac{\partial u}{\partial x}, \frac{\partial u}{\partial y}$ and $\mathbf{R}u$ are also solutions of $-\triangle v = \mu v$, where $\mathbf{R} = -y\frac{\partial}{\partial x} + x \frac{\partial}{\partial y}$ is the infinitesimal generator of action of rotation. The subspace we are interested in is given by applying $\frac{\partial}{\partial x}, \frac{\partial}{\partial y}$, and $\mathbf{R}$ repetitively on the solution $\omega$ of (\ref{pomp}-\ref{pomp3}) and then considering the linear combination of these functions.

This idea of using symmetry induced functions has been used successfully in comparing the relative magnitude of Dirichlet and Neumann eigenvalues of the Laplacian, see for example  Aviles (\cite{A}) or Levine and Weinberger (\cite{LW}). The novel part of our approach lies in an observation about the correspondence between the weak solution $u \in H^1(\Omega)$ of $-\triangle u = \mu u$ and its boundary data $(u|_{\partial \Omega}, \frac{\partial u}{\partial n}|_{\partial \Omega})$. As explained in (\cite{DJ}), if we denote the set of all weak solutions $u \in H^1(\Omega)$ which
satisfy (\ref{pomp}) as $K_{\mu}$, then the
well-known Green's formula,

\begin{equation}\label{green}
\int\int_{\Omega} \triangle v \cdot u  -  \triangle u \cdot v
\hspace{.1in} dxdy = \int_{\partial \Omega} \frac{\partial v}{\partial
n} \cdot u - \frac{\partial u}{\partial n} \cdot v \hspace{.1in} ds,
\end{equation}

\noindent  offers a connection between
the set of weak solutions $u$ of (\ref{pomp}) in $K_{\mu}$
and their boundary data $(u|_{\partial \Omega},
\frac{\partial u}{\partial n}|_{\partial \Omega})$ in the phase
space $H \stackrel{def}= H^{\frac{1}{2}}(\partial
\Omega) \times H^{-\frac{1}{2}}(\partial \Omega).$  This also prompts the following trace map on
$K_{\mu}$ via
$$T(u) = (u|_{\partial
\Omega}, \frac{\partial u}{\partial n}|_{\partial \Omega}),
\hspace{.1in} u \in K_{\mu}.$$

\noindent The trace map $T: K_{\mu} \te T(K_{\mu}) \subset H$ is an isomorphism, thus speaking roughly, the information of $u$ inside $\Omega$ is being transferred without loss to the boundary data by $T$. It is then convenient to work within the space $H$ consisting of boundary data, since two equations in (\ref{pomp}-\ref{pomp3}) are given in terms of boundary conditions.

\bigskip

Combining the two observations above, we are interested in the following question:\\

(*) If one applies $\frac{\partial}{\partial x}, \frac{\partial}{\partial y}, \mathbf{R}$ repetitively on the solution $\omega$ of (\ref{pomp}-\ref{pomp3}) and then applies the trace map $T$, what is the boundary behavior of these symmetry induced functions?

\bigskip

What we find is that the first few terms obtained this way  capture the "negative direction" on which the bilinear form $$B(\phi, \psi; \mu) = \int\int_{\Omega} \frac{\partial \phi}{\partial x} \cdot  \frac{\partial \bar \psi}{\partial x} + \frac{\partial \phi}{\partial y} \cdot  \frac{\partial \bar\psi}{\partial y}  - \mu \phi \cdot \bar\psi \hspace{.1in} dxdy, \hspace{.1in} \phi, \psi \in H^1(\Omega)$$  is semi-negative definite. Actually we only use terms no higher than second order, i.e. those terms like  $T(\omega_{xx}),  T(\mathbf{R}\omega_x) $ and lower order terms. Analysis of the bilinear form $B$ restricted to these terms gives the following main results in this paper :

\begin{theorem} Let $\Omega$ be an open, bounded domain in $R^2$ with connected and $C^{\infty}$ boundary, and $\omega$ a solution of
\begin{eqnarray}
  -\triangle \omega &=& \mu \omega \\
  \frac{\partial \omega}{\partial n}|_{\partial \Omega} &=& 0\\
  \omega|_{\partial \Omega} &=& const \neq 0
\end{eqnarray}

\noindent for some $\mu > 0$. If $\mu < \mu_8(\Omega)$, then $\Omega$ must be a disk. \end{theorem}

If we restrict to convex domain $\Omega$, a result of Brown and Kahane (\cite{BK}) established the Pompeiu property of $\Omega$ if the minimum diameter of $\Omega$ is less than or equal to half the maximum diameter. Also Aviles (\cite{A}) showed that if $\Omega$ is convex, then $\mu \leq \mu_7(\Omega)$ suffices to show that $\Omega$ is a disk. Thus basically our first result differs from that of Aviles by dropping the convexity condition on $\Omega$. For the convex case, we have the following

\begin{theorem} Let $\Omega$ be a strictly convex, centrally symmetric and bounded domain in $R^2$ with connected and $C^{\infty}$ boundary, and $\omega$ a solution of
\begin{eqnarray}
  -\triangle \omega &=& \mu \omega \\
  \frac{\partial \omega}{\partial n}|_{\partial \Omega} &=& 0\\
  \omega|_{\partial \Omega} &=& const \neq 0
\end{eqnarray}

\noindent for some $\mu > 0$. If $\mu < \mu_{13}(\Omega)$, then $\Omega$ must be a disk. \end{theorem}

\bigskip

The organization of the paper is as follows: in section $1.1$ we fix some notation and setting that will be used throughout the rest of the paper. In Section $2$ we
discuss properties of symmetry (translation and rotation) induced functions and their boundary behavior.
And finally in section $3$ we give proofs of two main results.

\subsection{Notation and Setting}

This section is to fix some notations and setup that will be used throughout the rest of the paper. First, denoting $0 = \mu_1(\Omega) < \mu_2(\Omega) \le \dots$ the set of all Neumann eigenvalues for the Laplacian on $\Omega$, and $0 < \lambda_1(\Omega) < \lambda_2(\Omega) \le \dots$ the set of all Dirichlet eigenvalues for the Laplacian on $\Omega$, with the associated Dirichlet eigenfunctions given by $u_1, u_2, \dots$.

Since $\omega$ is a solution for

\begin{eqnarray}
  -\triangle \omega &=& \mu \omega, \\
  \frac{\partial \omega}{\partial n}|_{\partial \Omega} &=& 0,\\
  \omega|_{\partial \Omega} &=& const \neq 0,
\end{eqnarray}

\noindent by rescaling $(x, y)  \te \sqrt{\mu}(x, y)$ we may assume that $\mu = 1$, and by multiplying an appropriate constant we may assume $\omega|_{\partial \Omega} \equiv 1$ .

Also for the domain $\Omega$, we will use arclength variable $s$ to parametrize $\partial \Omega$. Assuming that the total arclength of $\partial \Omega$ is $L > 0$, the boundary $\partial \Omega$ is given by the
parametric equation

\begin{equation}\label{parametrization_of_curve}
z(s) = x(s) + i y(s), s \in R
\end{equation}

\noindent where $z(s)$ is a $C^{\infty}$ and periodic function of $s$ of minimal period $L$.

Since $s$ is the arclength variable,  $\frac{dz}{ds}$ is a $C^{\infty}$ function of $s$ of unit length, thus we have

\begin{equation}
\frac{dz}{ds} = e^{i\theta(s)}, s \in R.
\end{equation}

\noindent where $\theta = \theta(s)$ is the angle of the tangent vector along $\partial \Omega$ with respect to $x$-axis, thus $\theta(s)$ is a $C^{\infty}$ function of $s$ with $\theta(s + L) - \theta(s) = 2\pi, s \in R$.

The curvature $\kappa(s)$ along $\partial \Omega$ is given by

\begin{equation}
\kappa(s) =  -\frac{d\theta}{ds},
\end{equation}

\noindent where the minus sign comes from the Frenet's theorem.


\begin{notation} We shall adapt the notation $\nabla \stackrel{def}= \frac{\partial }{\partial x} + i  \frac{\partial }{\partial y}, \bar \nabla \stackrel{def}= \frac{\partial }{\partial x} - i  \frac{\partial }{\partial y}$ instead of the usual notation $2 \frac{\partial}{ \partial \bar z} = \frac{\partial }{\partial x} + i  \frac{\partial }{\partial y}, 2 \frac{\partial}{\partial z} = \frac{\partial }{\partial x} - i  \frac{\partial }{\partial y}$.\end{notation}

Finally let $\vec{\xi}(s) = (\cos\theta(s), \sin\theta(s))$ be the unit tangent vector along $\partial \Omega$, and $\vec{n}(s) = (n_1, n_2)(s)$ be the unit outer normal vector along $\partial \Omega$, thus we have $(n_1, n_2)(s) = (\sin\theta(s), -\cos\theta(s))$ along $\partial \Omega$.

\section{Symmetry induced boundary value functions}

We first have the following

\begin{proposition}Let $u$ be a $C^{\infty}$ function in $\bar \Omega$ satisfying $-\triangle u = u$. Let $\phi = \nabla u$, then we have
the following commutative diagram:

$\begin{array}{ccc}

  u & \stackrel{\nabla}\longrightarrow & \phi \\
  T \downarrow &  & T \downarrow \\

  \left(
    \begin{array}{c}
      u|_{\partial \Omega} \\
      \frac{\partial u}{\partial n}|_{\partial \Omega} \\
    \end{array}
  \right)
  &
  \stackrel{M}\longrightarrow &  \left(
    \begin{array}{c}
      \phi|_{\partial \Omega} \\
      \frac{\partial \phi}{\partial n}|_{\partial \Omega} \\
    \end{array}
  \right)

   \end{array}$, where $M = e^{i\theta(s)} \cdot \left(
                                                  \begin{array}{cc}
                                                    \frac{d}{ds} & -i \\
                                                    \kappa\frac{d}{ds}+i(\frac{d^2}{ds^2}+ 1) & -i\kappa + \frac{d}{ds} \\
                                                  \end{array}
                                                \right)$.

\end{proposition}

\begin{proof} For any $z_0 = z(s_0) \in \partial \Omega$, we choose a new coordinate system $(\tilde x, \tilde y)$ such that the $\tilde x$-axis is along the direction of $\vec{\xi}(s_0) = (\cos\theta(s_0), \sin \theta(s_0))$, the $\tilde y$-axis is along the direction of $-\vec{n}(s_0) = (-\sin\theta(s_0), \cos\theta(s_0))$. Then we have

\begin{eqnarray}
e^{-i\theta(s_0)} \cdot \nabla u(z_0) = (\cos \theta(s_0) - i \sin \theta(s_0)) \cdot (u_x + i u_y)(z_0) = \\
. [\cos\theta(s_0) u_x(z_0) + \sin\theta(s_0) u_y(z_0)] - i [ \sin\theta(s_0) u_x(z_0) -\cos\theta(s_0) u_y(z_0)] \\
= \frac{d}{ds}|_{s = s_0} u|_{\partial \Omega} -i \frac{\partial u}{\partial n}|_{z = z_0},
\end{eqnarray}

\noindent which leads to

\begin{equation}\label{form1}
\nabla u(z_0) = e^{i\theta(s_0)} \cdot [ \frac{d}{ds}|_{s = s_0} u|_{\partial \Omega} -i \frac{\partial u}{\partial n}|_{z = z_0}]
\end{equation}

Secondly at $s = s_0$, we have

\begin{eqnarray*}
\frac{\partial \nabla u}{\partial n} = (\sin\theta, -\cos\theta) \cdot \left(
                                                  \begin{array}{cc}
                                                    u_{xx} & u_{xy} \\
                                                    u_{xy} & u_{yy}
                                                  \end{array}
                                                \right) \cdot \left(
                                                  \begin{array}{c}
                                                    1 \\
                                                    i
                                                  \end{array}
                                                \right) \\
= (\sin\theta, -\cos\theta) \cdot \left(
                                                  \begin{array}{cc}
                                                    -u_{yy} & u_{xy} \\
                                                    u_{xy} & -u_{xx}
                                                  \end{array}
                                                \right) \cdot \left(
                                                  \begin{array}{c}
                                                    1 \\
                                                    i
                                                  \end{array}
                                                \right) + e^{i\theta} \cdot i u \\
= (\sin\theta, -\cos\theta) \cdot \left(
                                                  \begin{array}{cc}
                                                    \cos \theta & -\sin\theta \\
                                                    \sin \theta & \cos \theta
                                                  \end{array}
                                                \right)  \left(
                                                  \begin{array}{cc}
                                                    -u_{\tilde y \tilde y} & u_{\tilde x \tilde y} \\
                                                    u_{\tilde x \tilde y} & -u_{\tilde x \tilde x}
                                                  \end{array}
                                                \right) \left(
                                                  \begin{array}{cc}
                                                    \cos \theta & \sin\theta \\
                                                    -\sin \theta & \cos \theta
                                                  \end{array}
                                                \right) \cdot \left(
                                                  \begin{array}{c}
                                                    1 \\
                                                    i
                                                  \end{array}
                                                \right) + e^{i\theta} \cdot i u \\
= e^{i\theta} \cdot i(u_{\tilde x \tilde x} + i u_{\tilde x \tilde y})  + e^{i\theta} \cdot  iu
\end{eqnarray*}

Note that

\begin{eqnarray} \label{frame1}
u_{\tilde x \tilde x} = \vec{\xi} \cdot \nabla^2 u \cdot \vec{\xi} = \frac{d^2 u}{ds^2} - \kappa \frac{\partial u}{\partial n}, \\ \label{frame2}
u_{\tilde x \tilde y} = -\vec{\xi} \cdot \nabla^2 u \cdot \vec{n} = -\frac{d \frac{\partial u}{\partial n}}{ds} - \kappa \cdot \frac{du}{ds}
\end{eqnarray}

\noindent where we used the Frenet's theorem $\frac{d \vec{\xi}}{ds} = \kappa \vec{n}, \frac{d \vec{n}}{ds} = -\kappa \vec{\xi}$.

Combining (\ref{frame1}), (\ref{frame2}) with the equation above we have that

\begin{equation}\label{form2}
\frac{\partial \nabla u}{\partial n} = e^{i\theta} \cdot [i(\frac{d^2 u}{ds^2 } + u) + \kappa \frac{du}{ds} + (-i \kappa + \frac{d}{ds})\frac{\partial u}{\partial n}]
\end{equation}

Thus from equations (\ref{form1}), (\ref{form2}) we complete the proof of the proposition.
\end{proof}

\begin{remark}
It is interesting to note that while $\nabla$ is a first-order partial differential operator, the corresponding matrix operator $M$ is a second order \emph{ordinary} differential operator that is purely "geometric", i.e., $M$ depends solely on $\theta = \theta(s)$.
\end{remark}

\begin{corollary}For $\bar \nabla = \frac{\partial}{\partial x} - i \frac{\partial}{\partial y}$, we have the corresponding matrix $\bar M = e^{-i\theta(s)} \cdot \left(
                                                  \begin{array}{cc}
                                                    \frac{d}{ds} & i \\
                                                    \kappa\frac{d}{ds}-i(\frac{d^2}{ds^2}+ 1) & i\kappa + \frac{d}{ds} \\
                                                  \end{array}
                                                \right)$ so that the following diagram commutes: $$\begin{array}{ccc}

  u & \stackrel{\bar \nabla}\longrightarrow & \psi \\
  T \downarrow &  & T \downarrow \\

  \left(
    \begin{array}{c}
      u|_{\partial \Omega} \\
      \frac{\partial u}{\partial n}|_{\partial \Omega} \\
    \end{array}
  \right)
  &
  \stackrel{ \bar M}\longrightarrow &  \left(
    \begin{array}{c}
      \psi|_{\partial \Omega} \\
      \frac{\partial \psi}{\partial n}|_{\partial \Omega} \\
    \end{array}
  \right)

   \end{array}$$ for any $C^{\infty}$ function $u$ in $\bar \Omega$ satisfying $-\triangle u = u$ and $\psi = \bar \nabla u$.

\end{corollary}

Instead of considering $\nabla = \frac{\partial}{\partial x} + i \frac{\partial}{\partial y}$ that corresponds to translational symmetry, one can also consider generators $\mathbf{R} \stackrel{def}= -y \frac{\partial}{\partial x} + x \frac{\partial}{\partial y}$ and $ \mathbf{S} \stackrel{def}= x \frac{\partial}{\partial x} + y \frac{\partial}{\partial y}$ that correspond to symmetry of rotation and scaling. Similar to Proposition 2.1 we have

\begin{corollary}For $\mathbf{R} + i\mathbf{S} = (-y + ix) \cdot \bar\nabla,$ we have the corresponding matrix $N = ( - y(s) + ix(s)  ) \bar M +  \left(
                                                  \begin{array}{cc}
                                                    0 & 0 \\
                                                    \frac{d}{ds} & i \\
                                                  \end{array}
                                                \right)$ (where $x(s) + i y(s)$ is the parametrization of $\partial \Omega$ with respect to arclength variable $s$), so that the following diagram commutes: $$\begin{array}{ccc}

  u & \stackrel{\mathbf{R} + i\mathbf{S}}\longrightarrow & \psi \\
  T \downarrow &  & T \downarrow \\

  \left(
    \begin{array}{c}
      u|_{\partial \Omega} \\
      \frac{\partial u}{\partial n}|_{\partial \Omega} \\
    \end{array}
  \right)
  &
  \stackrel{N}\longrightarrow &  \left(
    \begin{array}{c}
      \psi|_{\partial \Omega} \\
      \frac{\partial \psi}{\partial n}|_{\partial \Omega} \\
    \end{array}
  \right)

   \end{array}$$ for any $C^{\infty}$ function $u$ in $\bar \Omega$ satisfying $-\triangle u = u$ and $\psi = (\mathbf{R} + i\mathbf{S}) u$.

\end{corollary}

The approach we take is to apply $\nabla, \bar \nabla, \mathbf{R}$ repeatedly on $\omega$ to produce functions that lie in $K_{\mu}$, (i.e., they all satisfy $
-\triangle u = \mu u $) and then use the trace map $T$ on these functions to examine their boundary behavior. Using Proposition 2.1 and Corollary 2.4 we obtain easily the following

\begin{lemma} For $\frac{\partial \omega}{\partial x}, \frac{\partial \omega}{\partial y}, \frac{\partial^2 \omega}{\partial x^2}, \frac{\partial^2 \omega}{\partial x \partial y}, \frac{\partial^2 \omega}{\partial y^2}$, the trace map $T: K_{\mu} \te H^{\frac{1}{2}}(\partial \Omega) \times H^{-\frac{1}{2}}(\partial \Omega)$ on these functions gives the following table \\

\begin{tabular}{|c|c|c|c|c|c|}
               \hline
               $\backslash$ & $\frac{\partial \omega}{\partial x}$ & $\frac{\partial \omega}{\partial y}$ & $\frac{\partial^2 \omega}{\partial x^2}$ & $\frac{\partial^2 \omega}{\partial x \partial y}$ & $\frac{\partial^2 \omega}{\partial y^2}$ \\ \hline
               T  & $\left(
              \begin{array}{c}
                0 \\
                -\sin \theta(s)
              \end{array}
            \right)$ & $\left(
              \begin{array}{c}
                0 \\
                \cos \theta(s)
              \end{array}
            \right)$ & $\left(
              \begin{array}{c}
                -\frac{1}{2}(1 - \cos 2\theta(s)) \\
                \kappa\cos 2\theta(s)
              \end{array}
            \right)$ & $\left(
              \begin{array}{c}
                \frac{1}{2} \sin 2 \theta(s) \\
                \kappa \sin 2 \theta(s)
              \end{array}
            \right)$ & $\left(
              \begin{array}{c}
               -\frac{1}{2}( 1 + \cos 2\theta(s)) \\
               -\kappa \cos 2\theta(s)
              \end{array}
            \right)$ \\
               \hline
             \end{tabular}

\end{lemma} \bigskip

Similarly we have

\begin{lemma} For $\mathbf{R} \omega, \mathbf{R}^2 \omega, \nabla \mathbf{R}  \omega$, the trace map $T$ on these functions gives the following table \\

\begin{tabular}{|l|c|c|c|}
               \hline
               $\backslash$ & $\mathbf{R}\omega$ & $\mathbf{R}^2 \omega$ & $ \nabla \mathbf{R} \omega$ \\ \hline
               T   & $\left(
              \begin{array}{c}
                0 \\
                \frac{1}{2}\frac{d r^2}{ds}
              \end{array}
            \right)$   & $\left(
              \begin{array}{c}
                (-1) \cdot (\frac{1}{2} \frac{dr^2}{ds})^2 \\
                \frac{1}{2} \frac{d^2 r^2}{ds^2} \cdot (-y \frac{dx}{ds} + x \frac{dy}{ds}) -\kappa \cdot (\frac{1}{2} \frac{dr^2}{ds})^2
              \end{array}
            \right)$ & $\left(
              \begin{array}{c}
                (-i) \cdot e^{i\theta} \cdot \frac{1}{2} \frac{dr^2}{ds} \\
               (-i) \cdot e^{i\theta} \cdot ( \kappa \cdot \frac{1}{2} \frac{dr^2}{ds}  + i(\frac{1}{2}\frac{d^2 r^2}{ds^2}))
              \end{array}
            \right)$ \\
               \hline
             \end{tabular}

\noindent where $\theta = \theta(s), r^2(s) = x^2(s) + y^2(s)$.

\end{lemma}

\section{Proof of the main results}

\begin{theorem} Let $\Omega$ be an open, bounded domain in $R^2$ with connected and $C^{\infty}$ boundary, and $\omega$ a solution of
\begin{eqnarray}
  -\triangle \omega &=& \mu \omega \\
  \frac{\partial \omega}{\partial n}|_{\partial \Omega} &=& 0\\
  \omega|_{\partial \Omega} &=& const \neq 0
\end{eqnarray}

\noindent for some $\mu > 0$. If $\mu < \mu_8(\Omega)$, then $\Omega$ must be a disk. \end{theorem}

\begin{proof} We assume that $\mu < \mu_8(\Omega)$ and $\Omega$ is not a disk, first note that $\frac{\partial \omega}{\partial x}, \frac{\partial \omega}{\partial y}$ satisfies

\begin{eqnarray}
-\triangle u = \mu u, \hspace{.1in} z \in \Omega \\
u|_{\partial \Omega} = 0.
\end{eqnarray}

\noindent thus $\mu$ is also a Dirichlet eigenvalue for the Laplacian on $\Omega$. $\mu$ can not be $\lambda_1(\Omega)$ since $\frac{\partial \omega}{\partial x}$ changes sign in $\Omega$, due to Lemma 2.5. We claim that

\begin{claim} If $\Omega$ is not a disk, then $\mu > \lambda_2(\Omega) $.

\end{claim}

\begin{proof} If $\Omega$ is not a disk and we have $\mu = \lambda_2(\Omega) $, then $\mu$ as the second Dirichlet eigenvalue of Laplacian  has algebraic multiplicity at least three, with eigenfunctions given by $\frac{\partial \omega}{\partial x}, \frac{\partial \omega}{\partial y}, \mathbf{R} \omega$ (these eigenfunctions are linearly independent if $\Omega$ is not a disk, see (\cite{Pa}) for a proof). For any $c_1, c_2, c_3 \in R, c_1^2 + c_2^2 + c_3^2 = 1$, $c_1 \frac{\partial \omega}{\partial x} + c_2 \frac{\partial \omega}{\partial y} + c_3 \mathbf{R} \omega$ has exactly two nodal domains in $\Omega$, due to the Courant's nodal domain theorem. Also by Green's formula, we note that

\begin{equation}\label{est}
\left \{
\begin{array}{ll}
& \displaystyle \int_{\partial \Omega} \frac{\partial (c_1 \frac{\partial \omega}{\partial x} + c_2 \frac{\partial \omega}{\partial y} + c_3 \mathbf{R} \omega)}{\partial n} \cdot 1 ds  = \int_{\partial \Omega} \frac{\partial (c_1 \frac{\partial \omega}{\partial x} + c_2 \frac{\partial \omega}{\partial y} + c_3 \mathbf{R} \omega)}{\partial n} \cdot \omega|_{\partial \Omega} ds \\
& = \displaystyle \int_{\partial \Omega}  (c_1 \frac{\partial \omega}{\partial x} + c_2 \frac{\partial \omega}{\partial y} + c_3 \mathbf{R} \omega) \cdot \frac{\partial \omega}{\partial n}|_{\partial \Omega} ds = 0,
\end{array}\right.
\end{equation}

\noindent which implies $\frac{\partial (c_1 \frac{\partial \omega}{\partial x} + c_2 \frac{\partial \omega}{\partial y} + c_3 \mathbf{R} \omega)}{\partial n}$ has at least two zeros along $\partial \Omega$.

Now we will literally follow the line of proof given in Theorem 2.3 of Lin(\cite{Lin_CS}). Fix two points $P$ and
$P_i$ on $\partial \Omega$; we can always choose three constants $C_i^1, C_i^2, C_i^3$ such that

\begin{equation}\label{circle}
(C_i^1)^2 + (C_i^2)^2 +(C_i^3)^2 = 1,
\end{equation}

\noindent and the linear combination $\phi_i = C_i^1 \frac{\partial \omega}{\partial x}+  C_i^2 \frac{\partial \omega}{\partial y} + C_i^3 \mathbf{R} \omega$ satisfies $\nabla \phi_i(P) = \nabla \phi_i(P_i) = 0$. Note that $P$ and $P_i$ must be the only two zeros of $\frac{\partial \phi_i}{\partial n}$ on $\partial \Omega$.

Taking $P_i \te P$, and
by (\ref{circle}), there is a subsequence of $\phi_i$ which converges to $\phi$, and obviously $\phi \neq 0$ is
a second (Dirichlet) eigenfunction, given by the linear combination of $\frac{\partial \omega}{\partial x}, \frac{\partial \omega}{\partial y}$ and $\mathbf{R}\omega$, that $\frac{\partial \phi}{\partial n}$ has only one sign on $\partial \Omega$. But this contradicts with equation (\ref{est})!

\end{proof}

From the Claim 3.2 above we have that $\mu \ge \lambda_3(\Omega)$. Denoting $u_1, u_2$ the first and second Dirichlet eigenfunctions of Laplacian on $\Omega$ (there might be more than one eigenfunctions associated with $\lambda_2(\Omega)$, in that case we will choose any nonzero one), we define two subspaces $$W_1 \stackrel{def}= span \{u_1, u_2, \frac{\partial \omega}{\partial x}, \frac{\partial \omega}{\partial y}, \mathbf{R}\omega \}, \hspace{.2in} W_2 \stackrel{def}= span \{ \omega_{xx}, \omega_{xy}, \omega_{yy} \} ,$$ we note that $W_1 \cap W_2 = \{0\}$, and all functions in $W_1$ satisfy the Dirichlet boundary condition.

The bilinear form $$B(\phi, \psi; \lambda) = \int\int_{\Omega} \frac{\partial \phi}{\partial x} \cdot  \frac{\partial \bar \psi}{\partial x} + \frac{\partial \phi}{\partial y} \cdot  \frac{\partial \bar\psi}{\partial y}  - \lambda \phi \cdot \bar\psi \hspace{.1in} dxdy, \hspace{.1in} \phi, \psi \in H^1(\Omega), \lambda \in R$$ has the following property:

\begin{claim} $B(\cdot, \cdot; \mu)|_{W_1 \oplus W_2}$ is semi-negative definite.

\end{claim}

\begin{proof} The bilinear form $B$ restricted to $W_1$ is semi-negative definite, since all functions in $W_1$ correspond to linear combination of Dirichlet eigenfunctions of Laplacian with corresponding eigenvalues less than or equal to $\mu$. Also note that for any $\phi \in W_1, \psi \in W_2$, we have

\begin{equation}
B(\phi, \psi; \mu)  = \int_{\partial \Omega} \phi \cdot \frac{\partial \bar \psi}{\partial n} ds = 0
\end{equation}

\noindent where we have used the fact that $\phi \in W_1$ satisfies the Dirichlet boundary condition and $\psi \in W_2$ satisfies $-\triangle \psi = \mu \psi$.

It now suffices to show that $B$ restricted to $W_2$ is semi-negative definite. For this one note that for any $\psi = c_1 \omega_{xx} + c_2 \omega_{xy} + c_3 \omega_{yy} \in W_2, c_1, c_2, c_3 \in \mathcal{C}$, we have

\begin{equation}\label{est_W2}
\left \{
\begin{array}{ll}
& B(\psi, \psi; \mu)   =  \int_{\partial \Omega} \psi \cdot \frac{\partial \bar \psi}{\partial n} ds \\
& = -\frac{1}{2} \cdot \int_0^{2\pi} \lcb c_1 \cos 2 \theta + c_2 \sin 2\theta - c_3 \cos 2\theta \rcb \cdot \overline{\lcb  c_1 \cos 2 \theta + c_2 \sin 2\theta - c_3 \cos 2\theta \rcb} d \theta \le 0
\end{array}\right.
\end{equation}
\noindent where in the last equality we have used Lemma 2.5 and noted that the last integration in (\ref{est_W2}) does not depend on the parametrization $\theta = \theta(s)$. Thus we have completed the proof of Claim 3.3.

\end{proof}

Since $dim (W_1 \oplus W_2) = 8$, semi-negative definiteness of $B(\cdot, \cdot; \mu)$ on $W_1 \oplus W_2$ implies that $\mu \ge \mu_8(\Omega)$, due to the minimax principle of Neumann eigenvalues of Laplacian. But it contradicts with the assumption that $\mu < \mu_8(\Omega)$!
\end{proof}

\begin{theorem} Let $\Omega$ be a strictly convex, centrally symmetric and bounded domain in $R^2$ with connected and $C^{\infty}$ boundary, and $\omega$ a solution of
\begin{eqnarray}
  -\triangle \omega &=& \mu \omega \\
  \frac{\partial \omega}{\partial n}|_{\partial \Omega} &=& 0\\
  \omega|_{\partial \Omega} &=& const \neq 0
\end{eqnarray}

\noindent for some $\mu > 0$. If $\mu < \mu_{13}(\Omega)$, then $\Omega$ must be a disk. \end{theorem}

\begin{proof} Assuming that $\mu < \mu_{13}(\Omega)$ and $\Omega$ is not a disk, then first we have the following

\begin{claim} If $\Omega$ is strictly convex, centrally symmetric and non-disk, then $\mu > \lambda_5(\Omega)$.

\end{claim}

\begin{proof} We are going to examine the nodal line structure of $\mathbf{R} \omega$ and show that $\mathbf{R}\omega$ has at least $6$ nodal domains in $\Omega$, then by the Courant nodal domain theorem, we must have $\mu > \lambda_5(\Omega)$.

\begin{definition}  $\mathbf{N} = {\{(x, y) \in \bar \Omega | \mathbf{R}\omega(x, y) = 0 \}}$ is called the nodal line of $\mathbf{R}\omega$. We call $(x, y) \in \mathbf{N}$ a \emph{node} if $\nabla \mathbf{R} \omega(x, y) = 0$, i.e. where the gradient of $\mathbf{R}\omega$ vanishes.
\end{definition}

As for the distribution of nodes in $\Omega$, we claim that the origin is a node. This is because $\Omega$ is centrally symmetric, by unique continuation property (UCP) it follows that $\omega$ is an even function, i.e. $\omega(-x, -y) = \omega(x, y)$ for $(x, y) \in \Omega$.  $\mathbf{R}\omega$ is necessarily an even function on $\Omega$. It then follows that the origin is a node, since $\mathbf{R}\omega(0)$ as well as  $\nabla \mathbf{R} \omega(0) $ vanishes.

To study the distribution of nodes along $\partial \Omega$, using Lemma 2.5 and 2.6 we note that

\begin{equation}
0 = \int_{\partial \Omega} \mathbf{R}\omega \frac{\partial \omega_{xx}}{\partial n}    ds= \int_{\partial \Omega} \frac{\partial \mathbf{R}\omega}{\partial n}  \omega_{xx}  ds = \int_0^L \frac{\partial \mathbf{R}\omega}{\partial n}(s) \cdot  \frac{1}{2}(\cos2 \theta(s)-1)  ds,
\end{equation}

\noindent which gives rise to

\begin{equation}\label{root_of_Rw}
0 = \int_0^L \frac{\partial \mathbf{R}\omega}{\partial n}(s) \cdot  \cos 2\theta(s)  ds.
\end{equation}

 Since $\Omega$ is strictly convex, which implies that $\kappa = -\frac{d\theta}{ds} < 0$ along $\partial \Omega$, we may rewrite (\ref{root_of_Rw}) as

\begin{equation}\label{root_of_Rw_theta}
 \int_0^{2\pi} \frac{\partial \mathbf{R}\omega}{\partial n} \cdot  \cos 2\theta \cdot \frac{1}{\kappa} d \theta = 0,
\end{equation}

\noindent where instead of $s$ we are using $\theta$ as independent variable along $\partial \Omega$. Similarly using $ \int_{\partial \Omega} \frac{\partial \mathbf{R}\omega}{\partial n}  \omega  ds = 0,  \int_{\partial \Omega} \frac{\partial \mathbf{R}\omega}{\partial n}  \omega_{xy}  ds = 0$ we obtain

\begin{eqnarray}\label{root_of_Rw_theta_sin}
 \int_0^{2\pi} \frac{\partial \mathbf{R}\omega}{\partial n} \cdot   \frac{1}{\kappa} d \theta = 0, \\
 \int_0^{2\pi} \frac{\partial \mathbf{R}\omega}{\partial n} \cdot  \sin 2\theta \cdot \frac{1}{\kappa} d \theta = 0,
\end{eqnarray}

Also we should note that

\begin{eqnarray}\label{root_of_Rw_theta_odd_mode}
\int_0^{2\pi} \frac{\partial \mathbf{R}\omega}{\partial n} \cdot  \sin \theta \cdot \frac{1}{\kappa} d \theta = \int_0^{2\pi} \frac{\partial \mathbf{R}\omega}{\partial n} \cdot  \cos \theta \cdot \frac{1}{\kappa} d \theta  = 0, \\\label{root_of_Rw_theta_odd_mode2}
\int_0^{2\pi} \frac{\partial \mathbf{R}\omega}{\partial n} \cdot  \sin 3\theta \cdot \frac{1}{\kappa} d \theta = \int_0^{2\pi} \frac{\partial \mathbf{R}\omega}{\partial n} \cdot  \cos 3\theta \cdot \frac{1}{\kappa} d \theta =0
\end{eqnarray}

\noindent since  $\frac{1}{\kappa}, \frac{\partial \mathbf{R}\omega}{\partial n}$ remains invariant as $\theta \te \theta + \pi$.

Combining equations (\ref{root_of_Rw_theta})- (\ref{root_of_Rw_theta_odd_mode2}), using the Sturm's theorem that any smooth function of $\theta$ orthogonal to $\sin k\theta, \cos k\theta, 0 \le k \le N$ must have at least $2(N+1)$ roots in one period $2\pi$ of $\theta$(see for example Arnold (\cite{Arnold})), we have at least $8$ zeros for $\frac{\partial \mathbf{R}\omega}{\partial n}$ along $\partial \Omega$(note that $\kappa < 0$ along $\partial \Omega$). Since $\mathbf{R}\omega$ satisfies the Dirichlet boundary condition, existence of at least $8$ zeros of $\frac{\partial \mathbf{R}\omega}{\partial n}$ along $\partial \Omega$ implies that there exist at least $8$ nodes of $\mathbf{R}\omega$ along $\partial \Omega$.

The set of nodes is linked together by what we call segments. To be precise, $E \subset \mathbf{N}$ is called a \emph{segment} if $E$ is connected, there is no node inside $E$, and $\partial E$ consist of nodes.

To estimate the number of segments inside $\bar \Omega$, we denote the number of nodes inside $\Omega$ as $n_1$, the number of nodes along $\partial \Omega$ as $n_2$, then by the discussion above we have that $n_1 \ge 1, n_2 \ge 8$.

For any node $(x_0, y_0) \in \mathbf{N}$, the local behavior of nodal line $\mathbf{N}$ near $(x_0, y_0)$ is given by $p_n(x, y) + O(r^{n+\epsilon}), n \ge 2,$ where $p_n(x, y)$ is the $n$-th order spherical harmonic polynomial in $R^2$ (see e.g. Yau(\cite{Yau}), thus for each node inside $\Omega$ locally there exist at least $4$ segments associated with that node, and for each node on $\partial \Omega$ locally there exist at least $3$ segments(inside $\bar \Omega$) associated with that node. Since each of the segments estimated above has been counted twice, the total number of segments inside $\bar \Omega$, denoted by $S$, can be estimated by

\begin{equation}\label{est_s}
S \ge \frac{4n_1 + 3 n_2}{2} = 2n_1 + \frac{3}{2}n_2
\end{equation}

\noindent We will now differentiate between the following two cases:

\begin{itemize}

\item \noindent Case I) If the nodal line $\mathbf{N}$ is connected, then according to Euler's formula, the number of nodal domains inside $\Omega$, denoted by $D$, is given by

\begin{equation}
D = 1 + S - (n_1 + n_2) \ge 1 + n_1 + \frac{1}{2}n_2 \ge 1 + 1 + 4 = 6
\end{equation}

\noindent using the estimate (\ref{est_s}) above. \bigskip

\item \noindent  Case II) If the nodal line $\mathbf{N}$ is not connected, then applying Euler's formula to each connected components of $\mathbf{N}$, we have that the number of nodal domains inside $\Omega$, denoted by $D$, can be estimated by

\begin{equation}
D = C + S - (n_1 + n_2) \ge 1 + n_1 + \frac{1}{2}n_2 \ge 1 + 1 + 4 = 6
\end{equation}

\noindent where $C$ is the number of components of $\mathbf{N}$ and we are using the estimate (\ref{est_s}) in the first inequality above.

\end{itemize}

Now that $\mathbf{R}\omega$ has at least $6$ nodal domains in $\Omega$, then by the Courant's nodal domain theorem, we must have $\mu > \lambda_5(\Omega)$.

\end{proof}

From now on we are going to follow the line of proof of Theorem 3.1. By Claim 3.5 above we have that $\mu \ge \lambda_6(\Omega)$. Denoting $u_1, u_2, u_3, u_4, u_5$ the first five Dirichlet eigenfunctions of Laplacian on $\Omega$ (there might be more than one eigenfunctions associated with each $\lambda_i(\Omega), 2 \le i \le 5$, in that case we will choose any nonzero one), we define two subspaces $$W_1 \stackrel{def}= span \{u_1, u_2, u_3, u_4, u_5, \frac{\partial \omega}{\partial x}, \frac{\partial \omega}{\partial y}, \mathbf{R}\omega \}, \hspace{.2in} W_2 \stackrel{def}= span \{ \omega_{xx}, \omega_{xy}, \omega_{yy}, \nabla\mathbf{R} \omega, \bar\nabla \mathbf{R}\omega \} $$ we note that $W_1 \cap W_2 = \{0\}$ and the bilinear form $$B(\phi, \psi; \lambda) = \int\int_{\Omega} \frac{\partial \phi}{\partial x} \cdot  \frac{\partial \bar \psi}{\partial x} + \frac{\partial \phi}{\partial y} \cdot  \frac{\partial \bar\psi}{\partial y}  - \lambda \phi \cdot \bar\psi \hspace{.1in} dxdy, \hspace{.1in} \phi, \psi \in H^1(\Omega), \lambda \in R$$ has the following property:

\begin{claim} $B(\cdot, \cdot; \mu)|_{W_1 \oplus W_2}$ is semi-negative definite.

\end{claim}

\begin{proof} The bilinear form $B$ restricted to $W_1$ is semi-negative definite, since all functions in $W_1$ correspond to linear combination of Dirichlet eigenfunctions of Laplacian on $\Omega$ with corresponding eigenvalues less than or equal to $\mu$. Also note that for any $\phi \in W_1, \psi \in W_2$, we have

\begin{equation}
B(\phi, \psi; \mu)  = \int_{\partial \Omega} \phi \cdot \frac{\partial \bar \psi}{\partial n} ds = 0
\end{equation}

\noindent where we have used the fact that $\phi \in W_1$ satisfies the Dirichlet boundary condition and $\psi \in W_2$ satisfies $-\triangle \psi = \mu \psi$.

It now suffices to show that $B$ restricted to $W_2$ is semi-negative definite. For this let $V_1 = span \{ \omega_{xx}, \omega_{xy}, \omega_{yy} \}, V_2 = span \{ \nabla \mathbf{R} \omega, \bar \nabla \mathbf{R} \omega \}$. One note that $B$ restricted to $V_1$ is semi-negative definite, as already shown in Claim 3.3.

 Since $\Omega$ is centrally symmetric, we have $\theta(s + \frac{L}{2}) = \theta(s) + \pi, s \in R$. In that case by Lemma 2.5 and 2.6 we have

\begin{equation}\label{rem}
\int_{\partial \Omega} \frac{\partial \omega_{xx}}{\partial n} \cdot \nabla\mathbf{R}\omega ds = \int_0^L  \kappa \cos 2 \theta(s) \cdot (-i)e^{i\theta}\cdot \frac{1}{2}\frac{dr^2}{ds}  ds = 0,
\end{equation}

\noindent since the transformation $s \te s + \frac{L}{2}, \theta \te \theta + \pi$ (while $\kappa$ and $ \frac{dr^2}{ds}$ remain invariant) will reverse the sign of the last integral in (\ref{rem}). Similarly we have

\begin{eqnarray}\label{rem2}
\int_{\partial \Omega} \frac{\partial \omega_{xy}}{\partial n}\cdot \nabla \mathbf{R}\omega ds =
 \int_{\partial \Omega} \frac{\partial \omega_{yy}}{\partial n}\cdot \nabla \mathbf{R}\omega ds = \\
 \int_{\partial \Omega} \frac{\partial \omega_{xx}}{\partial n}\cdot \bar\nabla\mathbf{R}\omega ds =
 \int_{\partial \Omega} \frac{\partial \omega_{xy}}{\partial n}\cdot \bar\nabla\mathbf{R}\omega ds =
 \int_{\partial \Omega} \frac{\partial \omega_{yy}}{\partial n}\cdot \bar\nabla\mathbf{R}\omega ds  = 0.
\end{eqnarray}

From the discussion above we have that for any $\phi = c_1 \omega_{xx} + c_2 \omega_{xy} + c_3 \omega_{yy} \in V_1, \psi = c_4 \nabla\mathbf{R}\omega + c_5 \bar\nabla \mathbf{R}\omega \in V_2, c_i \in \mathcal{C}, 1 \le i \le 5 $, we have

\begin{equation}
B(\phi, \psi; \mu)  = \int_{\partial \Omega}  \frac{\partial  \phi}{\partial n} \cdot \bar \psi ds  = 0
\end{equation}
\noindent where we have used equations (\ref{rem}) and (\ref{rem2}). Thus we have shown that $V_1, V_2$ are orthogonal to each other as far as the bilinear form $B(\cdot, \cdot; \mu)$ is concerned.

Finally we show that the bilinear form $B$ restricted to $V_2$ is semi-negative definite, which is given by utilizing Lemma 2.6 to obtain

\begin{eqnarray}
B(\nabla\mathbf{R} \omega, \nabla\mathbf{R} \omega; \mu) = B(\bar\nabla\mathbf{R} \omega, \bar\nabla\mathbf{R} \omega; \mu) = - \int_0^{2\pi} (\frac{1}{2} \frac{dr^2}{ds})^2 d \theta \le 0 \\
B(\bar\nabla \mathbf{R} \omega, \nabla\mathbf{R} \omega; \mu) = 0
\end{eqnarray}

\noindent Thus we have completed the proof of Claim 3.7.

\end{proof}

Since $dim (W_1 \oplus W_2) = 13$, semi-negative definiteness of $B(\cdot, \cdot; \mu)$ on $W_1 \oplus W_2$ implies that $\mu \ge \mu_{13}(\Omega)$, due to the minimax principle of Neumann eigenvalues of Laplacian. But it contradicts with the assumption that $\mu < \mu_{13}(\Omega)$!

\end{proof}

\emph{\begin{ackname}{ \emph{The author would like to express his gratitude to the anonymous referee for the numerous comments and suggestions that greatly improve the paper.}}\end{ackname}}

\bibliographystyle{amsplain}

\end{document}